\documentclass[a4paper,oneside,10pt]{amsart}
\usepackage{color,amssymb,enumerate}

\usepackage[utf8]{inputenc}
\usepackage[T1]{fontenc}

\newcommand{\N}{\mathbb{N}}

\newcommand{\Prop}{\mathsf{Prop}}
\newcommand{\Form}{\mathsf{Form}}
\newcommand{\PForm}{\mathsf{PForm}}
\newcommand{\model}[1]{\mathcal{#1}}
\newcommand{\iimplies}{\rightarrow}
\newcommand{\struc}[1]{\langle #1 \rangle}
\newcommand{\Val}{\mathrm{Val}}
\newcommand{\MV}{\mathcal{MV}}
\DeclareSymbolFont{forpolishl}{T1}{cmr}{m}{n}
\DeclareMathSymbol{\lucas}{0}{forpolishl}{'212}
\newcommand{\Mod}{\mathrm{Mod}}
\newcommand{\tr}{\mathrm{tr}}

\newcommand{\idempo}[1]{\mathfrak{B}(#1)}
\newcommand{\logic}[1]{\mathsf{\mathbf{#1}}}

\newcommand{\lang}{\mathcal{L}}
\newcommand{\var}[1]{\mathcal{#1}}
\newcommand{\class}[1]{\var{#1}}
\newcommand{\alg}[1]{\mathrm{\bf #1}}

\newcommand{\classop}[1]{\mathbb{#1}}

\newcommand{\cate}[1]{\mathcal{#1}}
\newcommand{\framme}[1]{\mathfrak{#1}}

\newcommand{\bfphi}{\boldsymbol{\phi}}
\newcommand{\bfpsi}{\boldsymbol{\psi}}

\newcommand{\bfw}{\mathbf{w}}
\newcommand{\bfv}{\mathbf{v}}
\newcommand{\bfx}{\mathbf{x}}
\newcommand{\bfy}{\mathbf{y}}
\newcommand{\bfa}{\mathbf{a}}

\newcommand{\bfone}{\mathbf{1}}
\newcommand{\bfalpha}{\boldsymbol{\alpha}}
\newcommand{\disjun}{\biguplus}

\newcommand{\Ce}{\mathfrak{C}e}

\newcommand\restr[2]{{
  \left.\kern-\nulldelimiterspace 
  #1 
  \right|_{#2} 
  }}

\newtheorem{thm}{Theorem}[section]
\newtheorem{cor}[thm]{Corollary}
\newtheorem{lem}[thm]{Lemma}
\newtheorem{prop}[thm]{Proposition}
\theoremstyle{definition}
\newtheorem{defn}[thm]{Definition}
\newtheorem{ex}[thm]{Example}
\theoremstyle{remark}
\newtheorem{rem}[thm]{Remark}
\numberwithin{equation}{section}

\definecolor{gris10}{gray}{0.9}
\definecolor{gris}{gray}{0.5}

\title{Modal definability based on \L ukasiewicz  validity relations}

\author[B. Teheux]{Bruno Teheux}
\address{Mathematics Research Unit, FSTC, University of Luxembourg, 6, Rue Coudenhove-Kalergi, L-1359 Luxembourg, Luxembourg}
 \email{bruno.teheux@uni.lu}
\thanks{The author was supported  by the internal research project  F1R-MTH-PUL-15MRO3 of the University of Luxembourg.}
\date{\today}
\begin{document}
\begin{abstract}
We study two notions of definability for classes of relational structures based on modal extensions of \textsc{\L ukasiewicz} finitely-valued logics. The main results of the paper are the equivalent of the \textsc{Goldblatt}~-~\textsc{Thomason} theorem for these notions of definability.
\end{abstract}
\maketitle

\section{Introduction}

The language of modal logic is recognized as being efficient to talk about relational structures~-~for instance, see Slogan 1 in the preface of  \cite{Venema2001}. Actually, the connection between the modal language and relational structures is twofold. On the one hand, relational semantics help in the study of the \emph{deductive properties} of normal modal logics. The problem generally addressed is the following: given a modal logic $\logic{L}$, find a class of relational structures with respect to which $\logic{L}$ is complete. On the other hand, the modal language is used as a \emph{descriptive language}. There are at least two types of results to characterize the ability of the modal language to describe relational structures, in other words, to characterize its expressive power.

First, one can regard the expressive power of  the modal language as its ability to \emph{distinguish between worlds} in relational structures.  For example, the \textsc{van Benthem} theorem \cite{van_benthem1976} states that the modal language is the bisimulation-invariant fragment of first order logic.

A second approach to the expressive power of the modal language is to consider its  ability to \emph{distinguish between frames}, that is, its ability to define classes of frames. In this respect, one of the most renowned results is the \textsc{Goldblatt}~-~\textsc{Thomason} theorem \cite{Goldlblatt1975} that characterizes,  in terms of certain closure conditions, those first order definable classes of relational structures that are also definable by modal formulas. 

The notion of modal definability is based on the validity relation $\models$. This relation contains $(\framme{F}, \phi)$ (in notation $\framme{F}\models \phi$) whenever $\phi$ is a modal formula that is true in any model based on the structure $\framme{F}$. A class $\mathcal{C}$ of relational structures is \emph{modally definable} if there is a set $\Phi$ of modal formulas such that $\mathcal{C}=\{\framme{F} \mid \framme{F} \models \Phi\}$. Thus,  any change in the definition of  the validity relation affects the notion of modal definability.

One way to modify the validity relation is to tweak the definition of \emph{a model based on a structure}, that is, to change the set of possible valuations that can be added to a structure to turn it into a model. In  this paper, we study modal definability for validity relations defined with a notion of models in which formulas are evaluated in a finite set of truth-values using \textsc{\L ukasiewicz} interpretation of the connectors $\neg$ and $\iimplies$.     

The paper is organized as follows. In the second section, we introduce the modal language $\lang$ that we consider in the remainder of the paper, as well as two classes of relational structures  to interpret this language. The first one is the class of $\mathcal{L}$-frames and the second one is the class of $\mathcal{L}_n$-frames. The latter are $\lang$-frames in which the set of allowed  truth values is specified in each world.   These two classes of structures give rise to two validity relations, and to two corresponding notions of definability. In the third section, we develop some constructions  for $\lang$-frames and $\lang_n$-frames that can be used to test definability of classes of structures. The section is centered on the notion of canonical extension of structures. In the fourth section, we obtain many-valued versions of the \textsc{Goldblatt}~-~\textsc{Thomason} theorem, which constitute the main results of the paper (Theorems \ref{thm:main01} and \ref{thm:main02}). We conclude the paper by a section presenting some final remarks and topics for further research. Also, the proof of the Truth Lemma for $k$-ary modalities is given in an Appendix.

We use many results that were previously obtained  for modal extensions of \textsc{\L ukasiewicz} $(n+1)$-valued logic,  trying to avoid duplicating existing proofs by referring to the literature as far as possible without jeopardizing understanding. In particular, the present paper relies heavily on \cite{Teheux2012}, which should be considered as a companion paper. Moreover, we generalize some of the standard tools and techniques of Boolean modal logic, which eventually allows us to follow almost verbatim the original proof of   the \textsc{Goldblatt}~-~\textsc{Thomason} theorem to obtain Theorems \ref{thm:main01} and \ref{thm:main02}.

\section{From $\lucas_n$-valued models to definability}\label{sec:02}
Let $\lang=\{\neg, \iimplies, 1\}\cup\{\nabla_i \mid i\in I\}$ be a language, where $\neg$ is unary, $\iimplies$ is binary, $1$ is  constant and $\nabla_i$ is $k_i$-ary with $k_i\geq 1$ for every $i \in I$. Connectors in $\{\nabla_i\mid i \in I\}$ are considered as $k_i$-ary universal modalities. The set $\Form_\lang$ of formulas is defined by induction from an infinite set of propositional variables $\Prop$ using the grammar
\[
\phi::=p \ \vert\ 0 \ \vert\ \neg\phi\ \vert \ \phi\iimplies \phi \ \vert \ \nabla_i(\phi, \ldots, \phi),
\]
where $p \in \Prop$ and $i \in I$.
If no additional information is given, by \emph{a formula} we mean an element of $\Form_{\lang}$. We sometimes write $\phi(p_1, \ldots, p_k)$ to stress that $\phi$ is a formula whose propositional variables are among $p_1, \ldots, p_k$. 
When we write `let $\nabla$ be a $k$-ary modality' we mean `let $i\in I$  and $k\in \N$ such that $\nabla=\nabla_i$ and $k=k_i$'. In the examples, we often use the language $\lang_{\square}$ that contains only one  modal connector $\square$, which is unary. 

We use the customary abbreviations in \textsc{\L ukasiewicz} logic: we write $p\oplus q$ for $\neg p \iimplies q$, $p \odot q$ for $\neg (\neg p \oplus \neg q)$, $x\vee y$ for $(y\odot \neg x)\oplus x$, $x\wedge y$ for $(y \oplus \neg x)\odot x$, and $0$ for $\neg 1$. We assume associativity of $\oplus$ and $\odot$ and we denote by $k.p$ and $p^k$ the formulas $p\oplus \cdots \oplus p$ and $p\odot \cdots  \odot p$ (where $p$ is repeated $k$ times in both cases) for any $k\geq 0$.  We use bold letters to denote tuples (arity is given by the context). Hence, we denote by $\bfphi,\bfpsi,\ldots $  tuples of formulas  and by $\phi_i$ the $i$th component of $\bfphi$. 

To interpret formulas on structures, we use a many-valued generalization of the \textsc{Kripke} models. We fix a positive integer $n$ for the remainder of the paper and we denote by $\lucas_n$ the subalgebra $\lucas_n=\{0, \frac{1}{n}, \ldots, \frac{n-1}{n}, 1\}$ of the standard MV-algebra $\struc{[0,1],\neg,\iimplies, 1}$ which is defined by $x\iimplies y=\min(1, 1-x+y)$ and $\neg x= 1-x$. Hence, the interpretation of the connectors $\oplus$,  $\odot$, $\vee$ and $\wedge$ on $[0,1]$ are given by $x\oplus y=\min(x+y,1)$, $x\odot y= \max(x+y-1,0)$, $x\vee y=\max(x,y)$ and $x\wedge y=\min(x,y)$. 

Recall that the variety $\MV$ generated by the standard MV-algebra $[0,1]$ is the variety of \emph{MV-algebras} that was introduced by \textsc{Chang} \cite{Chang1958} in order to obtain an algebraic proof \cite{Chang1959} of the completeness of \textsc{\L ukasiewicz} infinite-valued logic with respect to $[0,1]$-valuations. We denote by $\MV_n$ the subvariety of $\MV$ generated by $\lucas_n$. For a general reference about the theory of MV-algebras, we refer to \cite{Cignoli2000}.

We  use vocabulary and notation that are customary for relational structures in the field of modal logic. Hence, an {$\lang$-frame} is a tuple $\framme{F}=\struc{W, (R_i)_{i\in I}}$ where $W$ is a nonempty set and $R_i$ is a $k_i+1$-ary relation for every $i \in I$. We say that $W$ is the \emph{universe} of $\framme{F}$ and that elements of $W$ are \emph{worlds} of $\framme{F}$. We denote by $\cate{FR}_{\lang}$ the class of $\mathcal{L}$-frames. For the sake of readability, we use $\cate{FR}$ for $\cate{FR}_{\lang_\square}$. By abuse of notation, we let $\lang$ stand for the language defined at the beginning of the section and for the signature of the $\lang$-frames. If $\ell\geq 2$ and $R\subseteq W^\ell$, we write $\bfw\in Ru$ for $(u, w_1, \ldots, w_{\ell-1})\in R$.

\begin{defn}\label{defn:mod}
An \emph{$\mbox{\L}_n$-valued $\lang$-model}, or a \emph{model} for short,  is a couple $\model{M}=\struc{\framme{F}, \Val}$ where $\framme{F}=\struc{W, (R_i)_{i\in I}}$ is an $\lang$-frame and  $\Val\colon W\times \Prop \to \lucas_n$. We say that  $\model{M}=\struc{\framme{F}, \Val}$ is \emph{based on $\framme{F}$}. 
\end{defn}

In a model $\mathcal{M}$, the valuation map $\Val$ is extended inductively  to $W \times \Form_{\lang}$ using \textsc{\L ukasiewicz} interpretation of the connectors $0$,  $\neg$ and $\rightarrow$ in $[0,1]$ and the rule
\[\Val(u,\nabla(\bfphi))=\min\{\max_{1 \leq \ell \leq k} \Val(w_\ell, \phi_\ell)\mid \bfw\in Ru\}\]
 for every $k$-ary modal connector $\nabla$.

Informally speaking, models have many-valued worlds and crisp relations. The class of $\lucas_n$-valued models has  been considered in  \cite{Bou2011,Hansoul2006,Teheux2012}  to obtain completeness results for many-valued normal modal logics.

\begin{defn}
A formula $\phi$ is \emph{true} in an $\lucas_n$-valued $\lang$-model $\model{M}=\struc{\framme{F}, \Val}$, in notation $\model{M} \models \phi$, if $\Val(u,\phi)=1$ for every world $u$ of  $\framme{F}$.


If $\Phi$ is a set of formulas that are true in any $\lucas_n$-valued $\lang$-model based on an $\lang$-frame $\framme{F}$, we write $\framme{F} \models_n \Phi$ and say that $\Phi$ is \emph{$\lucas_n$-valid} in $\framme{F}$. We write $\framme{F}\models_n \phi$ instead of $\framme{F}\models_n \{\phi\}$.
\end{defn}

We base our first notion of definability on the validity relation $\models_n$.
\begin{defn}
A class $\class{C}$ of $\lang$-frames is \emph{$\lucas_n$-definable} if there is a set  $\Phi$  of formulas   such that 
$\class{C}=\{\framme{F} \in \cate{FR}_{\lang}\mid \framme{F} \models_n \Phi\}$. In that case, we write $\class{C}=\mathrm{Mod}_n(\Phi)$.
\end{defn}

\begin{ex}\label{ex:jgy01}
As expected, the many-valued nature of the valuation added to the frames may be responsible for strong differences between the standard (Boolean) validity relation and the $\lucas_n$-validity relation. For instance $\Mod_1(\square (p \vee \neg p))=\class{FR}_{\lang_{\square}}$ while $\Mod_n(\square (p \vee \neg p))=\{\framme{F}\in\cate{FR}\mid R=\varnothing\}$ if $n>1$. 
\end{ex}

We denote by $\PForm_\lang^n$ the fragment of $\Form_\lang$ defined by the grammar
\[
\phi::=p^n \ \vert \ 0 \ \vert \ \neg \phi \ \vert  \phi \rightarrow \phi \ \vert \ \nabla_i(\phi, \ldots, \phi)
\]
where $p \in \Prop$ and $i\in I$. Let us also denote by $\mathrm{tr}_n$ the map  
 \[
\mathrm{tr}_n\colon \Form_{\lang} \to \PForm_{\lang}^n  \colon \phi(p_1, \ldots, p_k) \mapsto \phi(p_1^n, \ldots, p_k^n).
\]
The following result relates  the expressive power of $\Form_\lang$ with regards to $\models_1$ to the expressive power of $\PForm^m_\lang$ with regards to $\models_m$, for any $m> 0$.

\begin{prop}\label{prop:boh0}
Let $\class{C}$ be a class of $\lang$-frames and $\Phi\subseteq \Form_{\mathcal{L}}$. The following conditions are equivalent.
\begin{enumerate}[(i)]
\item\label{it:poi01} $\class{C}=\Mod_1(\Phi)$.
\item\label{it:poi02} There is an $m>0$ such that $\class{C}=\Mod_m(\mathrm{tr}_m(\Phi))$.
\item\label{it:poi03} For any $m>0$, we have $\class{C}=\Mod_m(\mathrm{tr}_m(\Phi))$.
\end{enumerate}
Moreover $\Mod_m(\Phi)  \subseteq  \Mod_1(\Phi)$ for every $m>0$.
\end{prop}
\begin{proof}
Obviously, (\ref{it:poi03}) implies (\ref{it:poi02}). Now, let $\framme{F}=\struc{W, (R_i)_{i \in I}}$ be an $\lang$-frame and $m>0$. We prove that for every $\phi \in \Form_\lang$ we have
\begin{equation}\label{eqn:hjl01}
\framme{F} \models_m \tr_m(\phi)\quad \iff \quad \framme{F} \models_1 \phi.
\end{equation}

For any $\Val\colon W\times \Prop \to \lucas_m$ let $\Val_\sharp \colon W\times \Prop \to \lucas_1$ be the map defined by $\Val_{\sharp}(u,p)=\Val(u, p^m)$ for every $u\in W$ and $p \in \Prop$. It is clear that
\begin{equation}\label{eqn:hjl02}
\lucas_1^{W \times \Prop}=\{\Val_\sharp \mid \Val \in \lucas_m^{W\times \Prop}\}.
\end{equation}
Also, by definition of the map $\tr_m$, for every $u\in W$ and $\phi \in \Form_\lang$ we have
\[
\Val(u,\tr_m(\phi))=\Val_\sharp(u,\phi).
\]
Then, for every  $\phi \in \Form_\lang$, $u\in W$ and $\Val\colon W \times \Prop \to \lucas_m$ we have
\[
\struc{\framme{F}, \Val} \models \tr_m(\phi) \iff \struc{\framme{F}, \Val_\sharp}\models \phi.
\]
We conclude that  \eqref{eqn:hjl01} holds true by the latter equivalence and identity \eqref{eqn:hjl02}.


We obtain directly from (\ref{eqn:hjl01})  that (\ref{it:poi02}) implies  (\ref{it:poi01}) and that (\ref{it:poi01}) implies (\ref{it:poi03}). The second part of the statement follows from the inclusion $\lucas_1 \subseteq  \lucas_m$ for every $m>0$. 
\end{proof}

In particular, any $\lucas_1$-definable class of frames is also $\lucas_m$-definable. At this point of our developments, nothing can be said about the converse implication (Theorem \ref{thm:main02} gives a partial answer).

Apart from the signature of $\lang$-frames, there is another first-order signature that can be used to interpret $\lang$-formulas. It is the signature  of the \emph{$\lang_n$-frames} that embodies the many-valued nature of the modal language we consider. These structures were introduced in \cite{Teheux2012,Teheux2009} and their  non-modal reducts were already defined in \cite{Niederkorn2001,Gehrke2004}. We denote by $\preceq$ the dual order of divisibility on $\N$, that is, for every $\ell, k \in \N$ we write $\ell \preceq k$ if $\ell$ is a divisor of $k$, and $\ell \prec k$ if $\ell$ is a proper divisor of $k$.

\begin{defn}[\cite{Teheux2012}]\label{def:enriched}
An \emph{$\lang_n$-frame} is a tuple $\struc{W, (r_m)_{m  \preceq n}, (R_i)_{i \in I}}$ where
\begin{enumerate}
\item $\struc{W, (R_i)_{i \in I}}$ is an $\lang$-frame, 
\item $r_m \subseteq W$ for every $m  \preceq  {n}$,
\item $r_n=W$ and $r_m \cap r_q=r_{\gcd(m,q)}$ for any $m, q  \preceq {n}$,
\item $R_i u \subseteq r_m^{k_i}$ for any $i \in I$,  $m   \preceq {n}$ and $u\in r_m$.
\end{enumerate}

We denote by $\class{FR}_\lang^n$ the class of the  $\lang_n$-frames and we write $\cate{FR}^n$ for $\cate{FR}_{\lang_\square}^n$. Elements of $\cate{FR}^n$ are called \emph{$n$-frames}. We use $\lang_n$ to denote the signature of the  $\lang_n$-frames. For $\framme{F}\in\class{FR}_\lang^n$, we let $\framme{F}_\sharp$ be the \emph{underlying $\lang$-frame of $\framme{F}$}, that is, the reduct of $\framme{F}$ to  the language of $\lang$-frames.
The \emph{trivial $\lang_n$-frame $\framme{F}^n$ associated with an $\lang$-frame $\framme{F}$} is  
obtained by enriching $\framme{F}$ with $(r_m)_{ m\preceq{n}}$ where $r_m=\varnothing$ if $m\neq n$ and $r_n=W$.
\end{defn}

The general idea is to use the structure given by the sets $r_m$ (where $m\preceq n$) to define a validity relation which is weaker than $\models_n$.  Informally, when  adding a valuation to an $\lang_n$-frame, we require that the truth value of any formula in any world $u\in r_m$ belongs $\lucas_m$. This idea is formalized in the following definition.

\begin{defn}\label{def:mobaseon}
An $\lucas_n$-valued model $\model{M}$ is \emph{based on} the  $\lang_n$-frame $\framme{F}=\struc{W, (r_m)_{ m  \preceq {n}}, (R_i)_{i \in I}}$ if $\model{M}$ is based on $\framme{F}_\sharp$ and $\Val(u, \Prop)\subseteq \lucas_m$ for every $m\preceq{n}$ and  $u \in r_m$. 

If $\Phi$ is a set of $\lang$-formulas that are true in every $\lucas_n$-valued model based on an $\lang_n$-frame $\framme{F}$, we write $\framme{F} \models \Phi$ and say that $\Phi$ is \emph{valid} in $\framme{F}$. We write $\framme{F}\models \phi$ instead of $\framme{F}\models \{\phi\}$.
\end{defn}

\begin{rem}\label{rem:nota}
It may seem counterintuitive to use the symbol $\models_n$ for the validity relation associated with $\lang$-frames and $\models$ for the one associated with $\lang_n$-frames (one could expect the reverse convention). Nevertheless, this is the sound way to use notation. Indeed, the signature of $\lang$-frames does not carry any information on the many-valued nature of the valuations that will be added on them to form models (hence the necessity to recall the dependence on $n$ in the validity relation related to $\lang$-frames), while this information is incorporated in the signature of $\lang_n$-frames. Keeping that remark in mind may help to remember the notation used in the paper.
\end{rem}

The structure given by the predicates $r_m$ ($m \preceq n)$ arises naturally in the algebraic treatment of (modal extensions of) \textsc{\L ukasiewicz} finitely-valued logics: it is used in the development of natural dualities for the variety $\var{MV}_n:=\classop{HSP}(\lucas_n)$  \cite{Niederkorn2001} and for the construction of canonical extensions of (modal expansions of) members of $\var{MV}_n$ \cite{Teheux2012, Gehrke2004}. Also, since the predicates $r_m$ ($m \preceq n)$ permit stating properties involving the set of truth degrees in each world, the gain of expressive power they provide could be interesting in applications of the modal extensions of \textsc{\L ukasiewicz} finitely-valued logic. For instance, in many-valued propositional dynamic logics \cite{Teheux2014}, it is interesting to be able to consider structures satisfying the formula `after an indefinite number of executions of the program $\alpha$ the formula $\phi$ is either true or false'. 

The proof of the following lemma is straightforward and omitted.

\begin{lem}\label{lem:trivial}
Let $\framme{F}$ be an $\lang$-frame and $\framme{F}^n$ be its associated trivial  $\lang_n$-frame.
\begin{enumerate}
\item \label{it:triv01} For every $\phi \in \Form_\lang$ we have $\framme{F} \models_n \phi$ if and only if $\framme{F}^n \models \phi$.
\item \label{it:triv02} For every $\phi \in \Form_\lang$, if $\framme{F} \models_n \phi$ then $\framme{G}\models \phi$ for every  $\lang_n$-frame based on $\framme{F}$.
\item \label{it:triv03} $(\framme{F}^n)_\sharp=\framme{F}$.
\end{enumerate}
\end{lem}

We use the validity relation $\models$ to introduce the notion of definability for  $\lang_n$-frames.

\begin{defn}
A class $\class{C}$ of  $\lang_n$-frames is \emph{definable} if there is some $\Phi\subseteq \Form_{\lang}$ such that 
$
\class{C}=\{\framme{F} \in \cate{FR}_{\lang}^n\mid \framme{F} \models \Phi\}.
$
In that case, we write $\class{C}=\Mod(\Phi)$.
\end{defn}
\begin{ex}\label{ex:jgy02}
It is not difficult to prove that $\Mod(\square (p \vee \neg p))=\{\framme{F} \in \cate{FR}^n \mid\forall u\,   Ru \subseteq r_1 \}$. Moreover, as we shall prove in Example \ref{ex:sfg}, the class $\{\framme{F}\in \cate{FR}^n \mid \forall u\, u \not\in r_m\}$ is not definable if $m$ is a strict divisor of $n$.
\end{ex}

Any  $\lang_n$-frame $\framme{F}$ has an underlying $\lang$-frame $\framme{F}_\sharp$. Conversely, for any $\lang$-frame $\framme{F}$, the trivial  $\lang_n$-frame $\framme{F}^n$ associated with $\framme{F}$ is based on $\framme{F}$. The following result clarifies the connections between these constructions with regards to definability.
\begin{prop}\label{prop:boh}
Let $\class{C}$ be a class of $\lang$-frames and $\Phi\subseteq \Form_\lang$. Denote by $\class{C}'$ the class $\{\framme{F}\in \cate{FR}_\lang^n \mid \framme{F}_\sharp \in \class{C}\}$. 
\begin{enumerate}
\item\label{it:biz01} If $\class{C}'=\Mod(\Phi)$ then $\class{C}=\Mod_n(\Phi)$.
\item\label{it:biz02} If $\class{C}=\Mod_n(\Phi)$ then $\class{C}'\subseteq \Mod(\Phi)$, but the converse inclusion may not hold.
\end{enumerate}
\end{prop}

\begin{proof}
(\ref{it:biz01}) If $\framme{F}\in \class{C}$ then any $\lucas_n$-valued $\lang$-model based on $\framme{F}$ can be viewed as an $\lucas_n$-valued $\lang$-model based on $\framme{F}^n$. By Lemma  \ref{lem:trivial} (\ref{it:triv01}), such an $\framme{F}^n$ belongs to $\class{C}'=\Mod(\Phi)$. If follows that $\framme{F}\in \Mod_n(\Phi)$ by definition of $\class{C}'$.

Conversely, if $\framme{F}\in \Mod_n(\Phi)$ then $\framme{F}^n \in \Mod(\Phi)=\class{C}'$ by Lemma \ref{lem:trivial} (\ref{it:triv01}). We conclude that $\framme{F}\in \class{C}$   by Lemma \ref{lem:trivial} (\ref{it:triv03}).

(\ref{it:biz02}) The stated inclusion follows directly from Lemma \ref{lem:trivial} (\ref{it:triv02}). To obtain a counterexample for the converse inclusion, assume that $n>1$ and consider the formula $\phi=\square(p \vee \neg p)$. We have stated in Example \ref{ex:jgy01} and Example \ref{ex:jgy02} that $\class{C}:=\Mod_n(\phi)$ is equal to $\{\framme{F}\in \cate{FR}\mid R=\varnothing\}$ and $\Mod(\phi)=\{\framme{F}\in \cate{FR}^n\mid \forall u \ Ru\subseteq r_1\}$. Hence, $\Mod(\phi)\not\subseteq \class{C}'$.
\end{proof}

\section{Testing definability with frame constructions}
There are several frame constructions that are known to preserve the standard Boolean validity relation. These constructions can be used to test if a class $\class{C}$ of frames is modally definable:  if $\class{C}$ is not closed under these constructions, it is not modally definable.

Three of them (namely, bounded morphisms, generated subframes and disjoint unions) admit straightforward many-valued versions. To deal with the fourth one, that is, canonical extension, we need to generalize some algebraic apparatus.

\subsection{$\lang_n$-bounded morphisms and generated subframes}\label{sect:constructionsclass}

If $R$ is a $(k+1)$-ary relation on a set $W$, if $u\in W$ and if $f\colon W\to W'$ is a map, we denote by $f(Ru)$ the set $\{(f(v_1), \ldots, f(v_k) )\mid \bfv \in Ru\}$. Recall that a map $f\colon\framme{F}\to \framme{F}'$ between two $\lang$-frames $\framme{F}=\struc{W, (R_i)_{i\in I}}$ and $\framme{F}'=\struc{W', (R'_i)_{i\in I}}$ is called a \emph{bounded morphism} if $f(R_iu)=R'_i f(u)$ for every world $u$ of $\framme{F}$ and $i\in I$. If in addition $f$ is onto, we write $f\colon \framme{F}\twoheadrightarrow\framme{F}' $ and say that $\framme{F}'$ is a \emph{bounded morphic image of $\framme{F}$}. 

A substructure $\framme{F}'$ of an $\lang$-frame $\framme{F}$ is a \emph{generated subframe} of $\framme{F}$, in notation $\framme{F}'\rightarrowtail \framme{F}$, if the inclusion map $i:\framme{F}'\to\framme{F}$ is a bounded morphism. 

\begin{defn}
A map $f:\framme{F}\to\framme{F}'$ between two  $\lang_n$-frames $\framme{F}$ 
and $\framme{F}'$
 is an \emph{$\lang_n$-bounded morphism} if $f$ is a bounded morphism between $\framme{F}_\sharp$ and $\framme{F}'_\sharp$ and if $f(r_m)\subseteq r_m'$ for every $m \preceq{n}$. If in addition $f$ is onto, we write $f\colon \framme{F}\twoheadrightarrow\framme{F}' $ and say that $\framme{F}'$ is an \emph{$\lang_n$-bounded morphic image of $\framme{F}$}.

A substructure $\framme{F}'$ of an  $\lang_n$-frame is an \emph{$\lang_n$-generated subframe} of $\framme{F}$, in notation $\framme{F}'\rightarrowtail \framme{F}$, if the inclusion map $i:\framme{F}'\to \framme{F}$ is a bounded morphism.
\end{defn}

If $\{\framme{F}_j \mid j \in J\}$ is a family of relational structures over the same signature (the signature of $\lang$-frames or the signature of  $\lang_n$-frames), we denote by $\disjun\{ \framme{F}_j\mid j \in J\}$ the disjoint union of these structures.

The next result shows how to use the constructions just introduced as \emph{criteria} for ($\lucas_n$-)definability. Proofs are routine arguments and are omitted.
\begin{prop}\label{prop:exo}
Let $\{\framme{F}, \framme{F}'\}\cup\{\framme{F}_j\mid j \in J\}$  be a family of $\lang$-frames,  $\{\framme{G}, \framme{G}'\}\cup\{\framme{G}_j\mid j \in J\}$  be a family of  $\lang_n$-frames, and $\phi\in \Form_\lang$.
\begin{enumerate}
\item\label{it:gss01} If $\framme{F}\models_n \phi$ and $\framme{F}'\rightarrowtail \framme{F}$ or $\framme{F} \twoheadrightarrow \framme{F}'$ then $\framme{F}'\models_n\phi$.
\item If $\framme{F}_j \models_n \phi$ for every $j\in J$ then $\disjun\{\framme{F}_j \mid j \in J\}\models_n \phi$.
\item If $\framme{G}\models \phi$ and $\framme{G}'\rightarrowtail \framme{G}$ or $\framme{G} \twoheadrightarrow \framme{G}'$ then $\framme{G}'\models\phi$.
\item If $\framme{G}_j \models \phi$ for every $j\in J$ then $\disjun\{\framme{G}_j \mid j \in J\}\models \phi$.
\end{enumerate}
\end{prop}
\begin{ex} \label{ex:sfg}
Assume that $n>1$ and that $k \prec n$, and set $\lang=\lang_\square$. The class $\class{C}_1=\{\framme{F}\in \cate{FR}^n \mid \forall u \ u \not \in r_k\}$ is not definable. Indeed, consider the two  $\lang_n$-frames $\framme{F}$ and $\framme{G}$ which both have an empty accessibility relation, whose universes are respectively $\{s\}$ and $\{t\}$ with $s\in r_\ell$ if and only if $\ell=n$  and $t\in r_\ell$ if and only if $k \preceq{\ell}$. Then the map $\framme{F} \to \framme{G}$ is an onto  $\lang_n$-bounded morphism, $\framme{F}\in \class{C}_1$ but $\framme{G}\not\in \class{C}_1$.

Similarly, the class $\class{C}_2=\{\framme{F}\in \cate{FR}^n \mid \exists u\ r_k \subseteq Ru\}$ is not definable. Indeed, consider the  $n$-frame $\framme{F}$ defined on the universe $\{u,v\}$ by setting $R=\{(u,v)\}$,   $u\in r_\ell$ if and only if $\ell=n$ and $v\in r_\ell$ if and only if $k\preceq{\ell}$. Then $\framme{F}\in \class{C}_2$ while it is not the case for the substructure $\framme{F}\vert_{\{v\}}$ which is an  $\lang_n$-generated subframe of $\framme{F}$.
\end{ex}

\begin{rem}\label{rem:notsimple}
Any $\class{C}\subseteq \cate{FR}_\lang$ is in bijective correspondence with the class $\class{C}^{n}$ of the trivial $\lang_n$-frames based on elements of $\class{C}$. Moreover, for any $\phi \in \Form_\lang$ and any $\framme{F}\in \class{C}$ we have $\framme{F}\models_n \phi$ if and only if $\framme{F}^n \models \phi$, which implies that the modal theory of $\class{C}$ and $\class{C}^n$ coincide. On the contrary, we cannot say that $\class{C}$ is modally definable if and only if so $\class{C}^n$ is. The simplest counterexample is given by taking $\lang=\lang_\square$ and $\class{C}=\cate{FR}$. We clearly have $\class{C}=\Mod_n(\varnothing)$ while $\class{C}^n$ is not definable according to Proposition \ref{prop:exo} (\ref{it:gss01}), as $\class{C}^n$ does not contain the $\lang_n$-bounded morphic images of its elements. As a result, the bijective correspondence between $\class{C}$ and $\class{C}^n$ does not permit deriving results about $\lucas_n$-definability of classes of $\lang$-frames from corresponding results about definability of classes of $\lang_n$-frames.
\end{rem}

\subsection{$\lang_n$-canonical extensions}

The most comfortable way to introduce canonical extensions of structures (Definitions \ref{defn:cano01} and \ref{defn:cano02}) is to go through the variety $\var{MMV}_n^\lang$ which is the algebraic counterpart of the modal extensions of \textsc{\L ukasiewicz} $n+1$-valued logics considered in  \cite{Hansoul2006,Teheux2012,Teheux2009}. 
In order to recall the definition of  $\var{MMV}_n^\lang$, we need to introduce some notation. For every $\bfx\in X^k$, every  $a\in X$, and  every $i\in\{1, \ldots, k\}$, we denote by $\bfx_i^a$ the $k$-tuple obtained from $\bfx$ by replacing $x_i$ with $a$.

 The variety $\var{MMV}_n^\lang$ is defined \cite{Teheux2012,Teheux2009} as the variety of $\lang$-algebras whose $\{\neg, \iimplies, 1\}$-reduct is an MV$_n$-algebra and that satisfy the equations
\begin{equation}\label{eq:dualop}
\nabla (\bfx_i^{y\iimplies z}) =\nabla(\bfx_i^y)\iimplies\nabla(\bfx_i^z), \qquad \nabla(\bfx \star \bfx)  =\nabla\bfx \star \nabla\bfx, \qquad \nabla(\bfx_i^1)  =1,
\end{equation}
for any $k$-ary modality $\nabla$,  $i\in\{1, \ldots, k\}$, and $\star \in\{\odot, \oplus\}$. A $k$-ary operation on an algebra $\alg{A}\in \MV_n$ that satisfies the equations in \eqref{eq:dualop} is called a \emph{$k$-ary modal operator}.

It follows that if $\alg{A}\in \var{MMV}_n^\lang$ then the Boolean algebra $\idempo{\alg{A}}$ of idempotent elements of $\alg{A}$ (\emph{i.e.}, elements $a\in \alg{A}$ that satisfy $a \oplus a =a$) equipped with the operations $\restr{\nabla}{\idempo{A}}$ for every modality $\nabla$ belongs to $\var{MMV}_1$ (which is the variety of Boolean algebras with $\lang$-operators). By abuse of notation, we denote this  algebra  by $\idempo{\alg{A}}$. Recall that an $\lang$-homomorphism $\alpha_\cdot\colon \Form_\lang \to \alg{A}$ where $\alg{A}\in \var{MMV}_n^\lang$ is called an \emph{algebraic valuation on $\alg{A}$}, and $\struc{\alg{A}, \alpha_\cdot}$ is called \emph{an algebraic model} (see \cite[Definition 4.4]{Teheux2012} and \cite[Definition 2.32]{Teheux2009}).

\begin{defn}\label{defn:pano01}
The \emph{canonical  $\lang_n$-frame associated with $\alg{A} \in \var{MMV}_n^\lang$}, in notation $\alg{A}^\times$, is the structure $\struc{W, (r_m)_{m \preceq n}, (R_i)_{i \in I}}$ whose universe is the set $W=\var{MV}(\alg{A}, \lucas_n)$ of MV-algebra homomorphisms from $\alg{A}$ to $\lucas_n$ and whose structure is defined by
\[
r_m=\var{MV}(\alg{A}, \lucas_m)\] 
 for every $m \preceq{n}$, and
\begin{equation}\label{eq:dual}
uR_i \bfv \quad \text{ if }\quad  \forall \bfa\in \alg{A}^{k_i}\, (u(\nabla_i\bfa)=1 \implies \max\{ v_\ell(a_\ell)\mid 1 \leq \ell \leq k_i\}=1),
\end{equation}
 for every $i \in I$.

The \emph{canonical $\lang$-frame associated with $\alg{A}\in \var{MMV}_n^\lang$}, in notation $\alg{A}^+$, is defined as $\alg{A}^+=(\alg{A}^\times)_\sharp$.

If $\struc{\alg{A},\alpha_\cdot}$ is an algebraic model, the \emph{canoncial model} of $\struc{\alg{A}, \alpha_\cdot}$ is the model $\struc{\alg{A}^+, \Val}$ defined by 
$\Val(u,p)=u(\alpha_p)$ for every $p \in \Prop$ and $u \in \alg{A}^+$.
\end{defn}

 Recall the following result, which states how any modal operator $\nabla_i$ of an $\var{MMV}_n^\lang$-algebra $\alg{A}$ can be recovered from its canonical relation $R_i$. The case of unary modalities was given in \cite[Proposition 5.6]{Teheux2012}. The proof for languages $\lang$ involving $k$-ary modalities with $k\geq 2$ was so far unpublished and is given  in the Appendix.

\begin{prop}\label{lem:truth}
If $\struc{\alg{A}, \alpha_\cdot}$ is an algebraic model with canonical model $\struc{\alg{A}^+, \Val}$, then
\[
\Val(u,\phi)=u(\alpha_\phi),
\]
for every $\phi\in \Form_\lang$ and $u\in \alg{A}^+$.
\end{prop}

This result can be used to prove that $\alg{A}^\times$ is an  $\lang_n$-frame for every $\alg{A}\in\var{MMV}_n^\lang$.

\begin{prop}\label{prop:is}
If $\alg{A}\in \var{MMV}_n^\lang$, then $\alg{A}^\times$ is an  $\lang_n$-frame. 
\end{prop}
\begin{proof}
We have to prove that for any $k$-ary modal operator $\nabla$ on an algebra $\alg{A}\in \var{MV}_n$, we have $R r_m \subseteq r_m^k$ for every $m \preceq n$, where $r_m$ and $R$ are defined on $\var{MV}(\alg{A}, \lucas_n)$ as in Definition \ref{defn:pano01}. If $\nabla$ is unary, the proof is provided in \cite[Lemma 7.4]{Teheux2012}. Let us prove the general case and assume that $k > 1$. For the sake of contradiction, suppose that $v_i \not\in r_m$ for some  $m\preceq n$, some  $u \in r_m$, some $\bfv \in Ru$ and some $i \leq k$. Let us denote by $\square$ the unary modal operator defined on $\alg{A}$ by $\square a=\nabla(\mathbf{0}_i^a)$ for every $a \in \alg{A}$. It follows from Proposition \ref{lem:truth} that for any $a\in A$ we have $u(\square a)=1$ if and only if $\min\{w_i(a)\mid (u,\bfw) \in R\}=1$. We deduce that if $u(\square a)=1$ then $v_i(a)=1$, which means that $(u,v_i)\in R_\square$ where $R_\square$ is the relation associated with $\square$ as in \eqref{eq:dual}. It follows that $v_i\in r_m$ since $R_\square$ is the relation associated with a unary modal operator on $\alg{A}$. This gives the desired contradiction.
\end{proof}



The following result identifies the canonical $\lang$-frame associated with $\mathcal{A}\in \var{MMV}_n$ with the canonical $\lang$-frame associated with its Boolean algebra of idempotent elements.


\begin{lem}[{\cite[Lemma 2.38]{Teheux2009}}]\label{lem:ba}
For every $\alg{A}\in \var{MMV}_n$, we have $\alg{A}^+\cong \idempo{\alg{A}}^+$ and an isomorphism is given by the map $u \mapsto \restr{u}{\idempo{\alg{A}}}$.
\end{lem}
\begin{proof}
It is known that the map $\psi$ defined on $\alg{A}^+$ by $\psi(u)=\restr{u}{\idempo{\alg{A}}}$ is a bijection between $\alg{A}^+$ and $\idempo{\alg{A}}^+$  \cite[Proposition 3.1]{Niederkorn2001}. It is also clear by definition \eqref{eq:dual} of the relation $R$ associated to a $k$-ary modal operator $\nabla$ that if $(u,v_1, \ldots, v_k)\in R^{\alg{A^+}}$ then $(\psi(u),\psi(v_1), \ldots, \psi(v_k))\in R^{\idempo{\alg{A^+}}}$. Conversely, let $(\psi(u),\psi(v_1), \ldots, \psi(v_k))\in R^{\idempo{\alg{A^+}}}$ and $(a_1, \ldots, a_k)\in \alg{A}^k$ be such that $u(\nabla(a_1, \ldots, a_k))=1$. It follows that $1=(u(\nabla(a_1, \ldots, a_k)))^{2n}=u(\nabla(a_1^{2n}, \ldots, a_k^{2n}))$ where $(a_1^{2n}, \ldots, a_k^{2n})\in \idempo{A}^{k}$. This means by definition of $R^{\idempo{A}^+}$ that $v_i(a_i^{2n})=1$ for some $i\leq k$, and hence that $v_i(a_i)=1$. 
\end{proof}


Now that we have a canonical way to associate structures to algebras of $\var{MMV}_n^\lang$, we consider the converse construction. That is, we define some ways to associate algebras to structures.
These constructions generalize the standard Boolean ones.

\begin{defn}[{\cite[ Definition 7.7]{Teheux2012}}]
The \emph{$\lucas_n$-complex algebra} of an $\lang$-frame $\framme{F}=\struc{W, (R_i)_{i \in I}}$ is the $\lang$-algebra
\[
\framme{F}_{+_n}=\struc{\lucas_n^W, \neg, \iimplies, 1, (\nabla_i)_{i \in I}}
\]
where $\neg$, $\iimplies$ and $1$ are defined componentwise and 
\[
\nabla_i \bfalpha (u)=\min \{\max_{1\leq \ell \leq k_i} \alpha_\ell(v_\ell)\mid \bfv \in R_i u\},
\]
for any modality $\nabla_i$, any $\bfx\in \lucas_n^W$, and  any $u\in W$.

The \emph{$\lucas_n$-tight complex algebra} of an  $\lang_n$-frame $\framme{F}=\struc{W, (r_m)_{ m  \preceq{n}}, (R_i)_{i\in I}}$ is the algebra
\[
\framme{F}_{\times}=\struc{\prod_{u\in W} \lucas_{s_u}, \neg, \iimplies, 1, (\nabla_i)_{i \in I}}
\]
where $s_u=\gcd\{m\preceq{n} \mid u \in r_m\}$ for every $u\in W$ and where the operations are defined  as for $\framme{F}_{+_n}$.
\end{defn}

 The following result, whose proof is routine, shows that  complex constructions give an algebraic translation of the validity relations. We use the standard equivalence between $ \lang$-formulas and $\lang$-terms.
\begin{prop}\label{prop:algtran}
Let $\phi \in \Form_\lang$.
\begin{enumerate}
\item\label{it:odfr02} If $\framme{F}$ is an $\lang$-frame then $\framme{F}_{+_n}\in \var{MMV}_n^\lang$, and $\framme{F} \models_n \phi$ if and only if $\framme{F}_{+_n} \models \phi = 1$.
\item If $\framme{F}$ is an  $\lang_n$-frame then  $\framme{F}_{\times}\in \var{MMV}_n^\lang$, and $\framme{F} \models \phi$ if and only if $\framme{F}_\times \models \phi =1$.
\end{enumerate}
\end{prop}


\begin{defn}\label{defn:cano01}
The \emph{canonical extension} of an  $\lang_n$-frame $\framme{F}$ is the structure $\Ce(\framme{F})=(\framme{F}_\times)^\times$.



\end{defn}

The notion of \emph{canonical extension}  \cite{Benthem1979,Venema2001,Goldlblatt1975} of  an $\lang$-frame $\framme{F}$,  also known as the \emph{ultrafilter extension}, is a classical tool in Boolean modal logic. It turns out that it is also relevant in our many-valued setting. It is convenient to adopt the following equivalent construction of this extension.
\begin{defn}\label{defn:cano02}
The \emph{canonical extension} of an $\lang$-frame $\framme{F}$ is the $\lang$-frame $\Ce(\framme{F})=(\idempo{\framme{F}_{+_n}})^+$.
\end{defn}

It is not difficult to check that $\idempo{\framme{F}_{+_n}} \cong \idempo{\framme{F}_{+_1}}$ for every $\lang$-frame $\framme{F}$. This isomorphism together with Lemma \ref{lem:ba} establish the equivalence between Definition \ref{defn:cano02} and the usual definition of the canonical extension of an $\lang$-frame.

\begin{lem}\label{lem:bgt}
If $\framme{F}$ is an  $\lang_n$-frame, then the map $f\colon \MV(\framme{F}_\times, \lucas_n) \to \MV(\idempo{\framme{F}_\times}, \lucas_n)\colon u \mapsto u|_{\idempo{\framme{F}_\times}}$ is an isomorphism between $\Ce(\framme{F})_\sharp$  and $\Ce(\framme{F}_\sharp)$.
\end{lem}
\begin{proof}
We obtain successively that
\[
\Ce(\framme{F})_\sharp=((\framme{F}_\times)^\times)_\sharp=(\framme{F}_\times)^+\cong\idempo{\framme{F}_\times}^+
=\idempo{\framme{F}_{+_n}}^+=\Ce(\framme{F}_\sharp),
\]
where the equalities follow from the definitions and the isomorphism is given by Lemma \ref{lem:ba}.
%
\end{proof}


We introduce the notion of $\lang_n$-canonical extension at the level of models.
\begin{defn}

The \emph{$\lang_n$-canonical extension} of a model $\model{M}=\struc{\framme{F}, \Val}$ is the $\lucas_n$-valued $\lang$-model  $\Ce_n(\model{M})=\struc{\Ce(\framme{F}), \Val^e}$  defined by setting $\Val^e(u,p)=u(\Val(-, p))$ for every $p \in \Prop$ and every world $u$ of $\Ce(\framme{F})$.

\end{defn}

To state the properties of the $\lang_n$-canonical extensions of the $\lucas_n$-valued $\lang$-models, we need to introduce the notion of \emph{submodel}.
\begin{defn}
Let $\model{M}=\struc{\framme{F}, \Val}$ and $\model{M}'=\struc{\framme{F}', \Val'}$ be two $\lucas_n$-valued $\lang$-models.  We say that $\model{M}$ is a \emph{submodel} of $\model{M}'$ if $\framme{F}$ is a substructure of $\framme{F}'$ and $\Val(u,p)=\Val'(u,p)$ for every world $u$ of $\framme{F}$ and  every $p \in \Prop$.
\end{defn}
 
\begin{prop}\label{prop:canonns}
Let $\model{M}=\struc{\framme{F},\Val}$ be an $\lucas_n$-valued $\lang$-model based on the $\lang_n$-frame $\framme{F}$. Denote by $\iota$  the map 
\[\iota\colon \framme{F}\to \Ce({\framme{F}})\colon w\mapsto \pi_w^{\framme{F}_\times}\] where $\pi_w^{\framme{F}_\times}$ denotes the projection map $\framme{F}_\times \to \lucas_{s_w}$ from $\framme{F}_\times $ onto its $w$-th factor. 
\begin{enumerate}
\item\label{it:bwz01} The map $\iota$ identifies $\framme{F}$ as a substructure of $\Ce(\framme{F})$.
\item\label{it:bwz02} The map $\iota$ identifies $\model{M}$ as a submodel of $\Ce_n(\model{M})$.
\end{enumerate}
\end{prop}
\begin{proof}
(\ref{it:bwz01}) It is known \cite[p. 95]{Venema2001} that the map
\[
\iota'\colon \framme{F}_\sharp \to \Ce(\framme{F}_\sharp)\colon w \mapsto \pi_w^{\idempo{\framme{F}_{\sharp +_n}}}
\]
identifies $\framme{F}_\sharp$ as a substructure of $\Ce(\framme{F}_\sharp)$. Using notation introduced in Lemma \ref{lem:bgt}, the map $f^{-1}\circ \iota'$ identifies $\framme{F}_\sharp$ as a substructure of $\Ce(\framme{F})_\sharp$. Since $f^{-1}\circ \iota'=\iota$, it remains to show that 
\[w\in r_m^{\framme{F}} \iff \iota(w) \in r_m^{\Ce(\framme{F})}\]
for every $m\preceq{n}$ and $w \in W$. This equivalence follows directly from the definition of $\iota$.

(\ref{it:bwz02}) For any world $w$ of $\model{M}$ and $p \in \Prop$, we obtain successively
\[
\Val^e(\iota(w),p)=\iota(w)(\Val(-,p))=\pi_w^{\framme{F}_\times}(\Val(-, p))=\Val(w, p),
\]
where the first equality is obtained by definition of $\Val^e$, the second and the third ones by definition of $\iota$ and $\pi_w^{\framme{F}_\times}$, respectively.
\end{proof}

\begin{prop}\label{prop:mdf}
Let $\model{M}=\struc{\framme{F}, \Val}$ be an $\lucas_n$-valued $\lang$-model based on the $\lang$-frame $\framme{F}$. For any world $u$ of $\Ce_n(\model{M})$ and any $\phi \in \Form_\lang$ we have $\Val^e(u, \phi)=u(\Val(-, \phi))$.
\end{prop}
\begin{proof}
The map $\alpha_\cdot\colon \Form_\lang \to \framme{F}_{+_n}$ defined as $\alpha_p=\Val(-,p)$ for every $p \in \Prop$ and extended as an $\lang$-homomorphism  is an algebraic valuation on $\framme{F}_{+_n}$. Moreover, $\Ce_n(\model{M})$ is the canonical model associated with the algebraic model $\struc{\framme{F}_{+_n}, \alpha_{\cdot}}$. It follows from Proposition \ref{lem:truth}  that $\Val^e(u, \phi)=u(\alpha_\phi)$, while equality $\Val(-,\phi)=\alpha_\phi$ holds by definition.
\end{proof}
\begin{cor}\label{cor:canon}
Let $\model{M}=\struc{\framme{F}, \Val}$ be an $\lucas_n$-valued $\lang$-model. For every world $u$ of $\model{M}$ and every $\phi \in \Form_\lang$, we have $\Val^e(\iota(u),\phi)=\Val(u,\phi)$, where $\iota$ is the map defined in Proposition \ref{prop:canonns}.
\end{cor}
\begin{proof}
We have 
\[
\Val^e(\iota(u), \phi)=\iota(u)(\Val(-,\phi))=\Val(u, \phi),
\]
where the first equality is obtained by Proposition \ref{prop:mdf} and the second by definition of $\iota$.
\end{proof}

From Corollary \ref{cor:canon} we obtain the fact that canonical extensions of structures reflect the validity relation, as stated in the next result.
\begin{cor}\label{cor:canon02}
Let $\framme{F}$ and $\framme{G}$ be an $\lang$-frame and an  $\lang_n$-frame, respectively, and $\phi$ be a formula.
\begin{enumerate}
\item\label{it:ldsg01} If $\Ce(\framme{G}) \models  \phi$ then $\framme{G} \models \phi$.
\item\label{it:ldsg02} If $\Ce(\framme{F})\models_n \phi$ then $\framme{F} \models_n \phi$.
\end{enumerate}
\end{cor}
\begin{proof}
(\ref{it:ldsg01}) follows directly from Corollary \ref{cor:canon}. For (\ref{it:ldsg02}), first note that
\[
\Ce(\framme{F}) = \Ce((\framme{F}^n)_\sharp) = \Ce(\framme{F}^n)_\sharp,
\]
where the first equality is obtained by Lemme \ref{lem:trivial} (\ref{it:triv03}) and the second one by Lemma \ref{lem:bgt}. Thus, $\Ce(\framme{F}^n)$ is an  $\lang_n$-frame based on $\Ce(\framme{F})$. If follows by Lemma \ref{lem:trivial} (\ref{it:triv02}) that if $\phi\in \Form_\lang$ is such that $\Ce(\framme{F})\models_n \phi$, then $\Ce(\framme{F}^n) \models \phi$. Then, we obtain $\framme{F}^n \models \phi$ by statement (\ref{it:ldsg01}), or equivalently that $\framme{F}\models_n \phi$ by Lemma   \ref{lem:trivial} (\ref{it:triv01}).
\end{proof}

\section{\textsc{Goldblatt}~-~\textsc{Thomason} Theorems}

We pursue the algebraic approach to frame definability in our proofs of the \textsc{Goldblatt}~-~\textsc{Thomason} theorems. The proofs rely on two ingredients: a correspondence between construction operators for algebras and frames, and a construction of the  canonical extensions and $\lang_n$-canonical extensions of $\lang$-frames and $\lang_n$-frames, respectively, as ultrapowers. Regarding the first ingredient, we only expose the tools needed for our purpose, without developing a real duality. Our approach is a step-by-step adaptation of the original proof of the \textsc{Goldblatt}~-~\textsc{Thomason} Theorem \cite{Goldlblatt1975}.

\begin{prop}\label{prop:sum}
If $\{\framme{F}_i\mid i \in I\}$ is a family of  $\lang_n$-frames then $(\disjun_{i \in I} \framme{F}_i)_\times$ is isomorphic to $\prod_{i \in I} (\framme{F}_i)_\times$. In particular, if $\{\framme{F}_i \mid i \in I\}$ is a family of $\lang$-frames then $(\disjun_{i \in I} \framme{F}_i)_{+_n}$ is isomorphic to $\prod_{i \in I} (\framme{F}_i)_{+_n}$.
\end{prop}
\begin{proof}
The map $f\colon (\disjun_{i \in I} \framme{F}_i)_\times \to \prod_{i \in I} (\framme{F}_i)_\times$ defined by $f(\alpha)_i(u)=\alpha(u)$ for every $i \in I$, $\alpha \in (\disjun_{i \in I} \framme{F}_i)_\times$, and $u \in \framme{F}_i$ is an isomorphism.
\end{proof}

\begin{prop}\label{prop:dsffg}
If $f\colon \alg{A} \to \alg{A}'$ is an $\var{MMV}_n$-homomorphism between two $\var{MMV}_n$-algebras $\alg{A}$ and $\alg{A}'$, then the map $f^\times\colon \alg{A}'{}^\times \to \alg{A}^\times\colon u \mapsto u \circ f$ is an  $\lang_n$-bounded morphism. In particular, it is a bounded $\lang$-morphism from $\alg{A}'{}^+ \to \alg{A}^+$. 

In addition, if $f$ is one-to-one then $f^\times$ is onto. If $f$ is onto then $f^\times$ is one-to-one.
\end{prop}
\begin{proof}
From Lemma \ref{lem:ba} and \cite[Theorem 3.2.4]{Goldblatt1989} we first obtain that $f^\times$ is a bounded $\lang$-morphism from $(\alg{A}'{}^\times)_\sharp\cong\alg{A}'{}^+$ to $(\alg{A}^\times)_\sharp\cong\alg{A}^+$. Moreover, $f^\times$ clearly satisfies $f(u) \in r_m^{\alg{A}^\times}$ for every $u \in r_m^{\alg{A}'{}^\times}$.

The second part of the proof follows once again by Lemma \ref{lem:ba} and \cite[Theorem 3.2.4]{Goldblatt1989}.
\end{proof}

Theorem \ref{thm:canoisult} lifts the following known result at the level of  $\lang_n$-frames.
 \begin{thm}[{\cite[Theorem 3.6.1]{Goldblatt1989}}]\label{thm:canoisult00}
The canonical extension of an $\lang$-frame $\framme{F}$ is a bounded morphic image of an ultrapower of $\framme{F}$.
\end{thm}

\begin{thm}\label{thm:canoisult}
The canonical extension of an  $\lang_n$-frame $\framme{F}$ is an  $\lang_n$-bounded morphic image of an ultrapower of $\framme{F}$.
\end{thm}

\begin{proof}
We adapt the proof of the corresponding result for the class of $\lang$-frames given in \cite[Theorem 3.6.1]{Goldblatt1989}. Denote by $\lang_{\framme{F}}$ the language $\lang \cup \{P_X \mid X \subseteq W\}$ where $W$ is the universe of $\framme{F}$ and $P_X$ is a unary predicate for every $X \subseteq W$. We turn $\framme{F}$ into an $\lang_{\framme{F}}$-structure $\framme{F}'$ by setting $\framme{F}', w \models P_X$ if $w \in X$, for any $w \in W$ and $X \subseteq W$. Theorem 6.1.8 in \cite{Chang1990} provides an $\omega$-saturated ultrapower $\framme{F}_\omega$ of $\framme{F}'$. We prove that
\begin{equation}\label{eqn:ultra}
\Ce(\framme{F}) \text{ is an } \lucas_n \text{-valued bounded morphic image of the } \lang \text{-reduct of } \framme{F}_\omega.
\end{equation}
It is shown in \cite[Theorem 3.6.1]{Goldblatt1989} that for every element $x$ of $\framme{F}_\omega$ the set $F_x=\{X \subseteq W \mid \framme{F}_\omega, x \models P_X\}$ is an ultrafilter of $2^W=\idempo{\framme{F}_{\times}}$. Thus, for every $x \in \framme{F}_\omega$, there is a unique $u_x\in \MV(\framme{F}_\times, \lucas_n)$ which satisfies $F_x=u_x^{-1}(1) \cap \idempo{\framme{F}_{\times}}$. To obtain \eqref{eqn:ultra}, we prove that the map $f\colon \framme{F}_\omega \to \Ce(\framme{F})\colon x \mapsto u_x$ is an $\lang_n$-bounded morphism. It is shown in \cite[Theorem 3.6.1]{Goldblatt1989} that $f$ is a bounded morphism from $(\framme{F}_\omega)_\sharp$ onto $\Ce(\framme{F})_\sharp\cong\Ce(\framme{F}_\sharp)$. It remains to prove that $f(r_m^{\framme{F}_\omega}) \subseteq r_m^{\Ce(\framme{F})}$ for every $m \preceq{n}$. Let $x\in r_m^{\framme{F}_\omega}$ and $\alpha \in \framme{F}_\times$. We have to prove that $u_x(\alpha) \in \lucas_m$, or equivalently that $u_x(I_{m}(\alpha))=1$ where $I_m$ is an MV-term whose interpretation on $\lucas_n$ is valued in $\{0,1\}$ and satisfies
 \[
I_m^{\lucas_n}(a)=1 \iff a \in \lucas_m,
\]
 for every $a\in \lucas_n$  (the existence of such a term is a consequence of the McNaughton Theorem \cite{McNaughton1951}). If $X_{\alpha, m}$ denotes the set $\{y \in W \mid I_m(\alpha)(y)=1\}$ then $\framme{F}' \models \forall v (v \in r_m \Rightarrow P_{X_{\alpha, m}}(v))$ by definition of $\framme{F}_\times$, from which we deduce $\framme{F}_\omega \models \forall v (v \in r_m \Rightarrow P_{X_{\alpha, m}}(v))$ since $\framme{F}_\omega$ is an elementary extension of $\framme{F}'$. It follows by definition of $u_x$ that $u_x(I_m(\alpha))=1$, and we have proved \eqref{eqn:ultra}. 
\end{proof}

\begin{rem}\label{rem:bir}
Recall that, thanks to the \textsc{Birkhoff}  theorem on varieties, if  $K\cup \{\alg{A}\}$ is a class of algebras of the same type, then $\alg{A}$ belongs to the equational class defined by the equational theory of $K$ if and only if $\alg{A}\in \classop{HSP}(K)$. 
\end{rem}

We have gathered the tools needed to obtain the $\lucas_n$-valued versions of the \textsc{Goldblatt}-\textsc{Thomason} theorem.

\begin{thm}\label{thm:main01}
Assume that $\mathcal{C}$ is a class of  $\lang_n$-frames that contains  ultrapowers of its elements. Then $\mathcal{C}$ is definable if and only if the following two conditions are satisfied.
\begin{enumerate}
\item\label{it:ljd01} The class $\mathcal{C}$ contains  $\lang_n$-generated subframes, disjoint unions and  $\lang_n$-bounded morphic images of its elements.
\item\label{it:ljd02} For any  $\lang_n$-frame $\framme{F}$, if $\Ce(\framme{F}) \in \mathcal{C}$ then $\framme{F}\in \mathcal{C}$. 
\end{enumerate}
\end{thm}

\begin{thm}\label{thm:main02}
Assume that $\mathcal{C}$ is a class of $\lang$-frames that contains  ultrapowers of its elements. Then $\mathcal{C}$ is $\lucas_n$-definable if and only if the following two conditions are satisfied.
\begin{enumerate}
\item\label{it:ljd03} The class $\mathcal{C}$ contains generated subframes, disjoint unions and bounded morphic images of its elements.
\item\label{it:ljd04} For any $\lang$-frame $\framme{F}$, if $\Ce(\framme{F}) \in \mathcal{C}$ then $\framme{F}\in \mathcal{C}$. 
\end{enumerate}
\end{thm}
\begin{proof}[Proof of Theorem \ref{thm:main01}] 

Necessity follows from Proposition \ref{prop:exo} and Corollary \ref{cor:canon02}. For sufficiency, suppose that $\mathcal{C}$ is a class  of  $\lang_n$-frames that satisfies conditions (\ref{it:ljd01}) and (\ref{it:ljd02}) of Theorem \ref{thm:main01}. Let $\Lambda_{\mathcal{C}}$  be the set of $\lang$-formulas defined as
\[
\Lambda_{\mathcal{C}} =\bigcap_{\framme{F}\in \mathcal{C}}\{\phi  \in \Form_\lang \mid \framme{F} \models \phi \}.
\]
We prove that $\mathcal{C}=\Mod(\Lambda_\mathcal{C})$.
We have $\mathcal{C} \subseteq\Mod(\Lambda_\mathcal{C})$ by definition of $\Lambda_\mathcal{C}$. For the other inclusion, let $\framme{F}\in\Mod(\Lambda_\mathcal{C})$. From Proposition \ref{prop:algtran} (\ref{it:odfr02}), it follows that $\framme{F}_\times$ satisfies every equation that is satisfied by every member of $\mathcal{C}_\times=\{\framme{F}_\times \mid \framme{F}\in \mathcal{C}\}$. We deduce from Remark \ref{rem:bir} that $\framme{F}_\times\in \classop{HSP}(\mathcal{C}_\times)$, and there exist a family $ \{\framme{F}_i\mid i \in I\}$ of elements of $\mathcal{C}$ and a subalgebra $\alg{A}$ of $\prod_{i\in I} \framme{F}_{i\times}$ such that $\framme{F}_\times$ is a homomorphic image of $\alg{A}$. By considering the canonical structures associated to these algebras, we obtain  by  Proposition \ref{prop:dsffg} that
\[
\big(\big(\disjun_{i \in I} \framme{F}_{i}\big)_\times\big)^\times \twoheadrightarrow \alg{A}^\times \hookleftarrow (\framme{F}_\times)^\times.
\]
From our assumptions on $\mathcal{C}$ and Theorem \ref{thm:canoisult}, we obtain that $\big(\big(\disjun_{i \in I} \framme{F}_{i}\big)_\times\big)^\times$ belongs to $\mathcal{C}$, thus so do $\alg{A}^\times$ and $(\framme{F}_\times)^\times$. We conclude that $\framme{F}\in \mathcal{C}$ using assumption (\ref{it:ljd04}) since $(\framme{F}_\times)^\times=\Ce(\framme{F})$.  
\end{proof}

\begin{proof}[Proof of Theorem \ref{thm:main02}] Necessity follows from Proposition  \ref{prop:exo} and Corollary \ref{cor:canon02}. We obtain sufficiency as  a corollary of Theorem \ref{thm:main01}. Let $\mathcal{C}$ be a class of $\lang$-frames that contains ultrapowers of is elements and that satisfies conditions (\ref{it:ljd03}) and (\ref{it:ljd04}) of Theorem \ref{thm:main02}. Then the class $\mathcal{C'}=\{\framme{F}\in \class{FR}_\lang^n \mid \framme{F}_\sharp \in \mathcal{C}\}$ contains ultrapowers of its elements and satisfies conditions (\ref{it:ljd01}) and (\ref{it:ljd02}) of Theorem \ref{thm:main01}. It follows that there is a set $\Lambda$ of $\lang$-formulas such that $\mathcal{C'}=\Mod(\Lambda)$. We obtain that $\mathcal{C}=\Mod_n(\Lambda)$ by Proposition \ref{prop:boh} (\ref{it:biz01}).
\end{proof}

Note that by definition of the frame constructions,  Theorem  \ref{thm:main02} can be equivalently restated as follows.
\begin{cor}\label{cor:godequiv}
Assume that $\mathcal{C}$ is a class of $\lang$-frames that contains  ultrapowers of its elements. Then $\mathcal{C}$ is $\lucas_n$-definable if and only it is $\lucas_1$-definable.
\end{cor}

\section{Conclusions and further research}\label{sec:last}
The results obtained in this paper clarify some links between the standard  notion of modal definability and two of its generalizations based on \L ukasiewicz $(n+1)$-valued logic. We conclude the paper by presenting some final remarks and topics for further research.

\begin{enumerate}[(I)]
\item\label{it:deci} Theorem \ref{thm:main02} completely deciphers the links between standard  modal definability and $\lucas_n$-valued definability for elementary classes of $\lang$-frames. Indeed, as a corollary of Theorem \ref{thm:main02} and the \textsc{Goldblatt}~-~\textsc{Thomason} theorem \cite{Goldlblatt1975}, we obtain that those elementary classes of $\lang$-frames that are $\lucas_n$-definable are exactly the ones that are modally definable. Deciphering these links in the non-elementary case is a topic of interest, and Propositions \ref{prop:boh0} and \ref{prop:boh} can be considered as modest steps towards some solution to this problem. In particular, it is an open problem to find a non-elementary class of $\lang$-frames which is $\lucas_n$-definable without being $\lucas_1$-definable.

\item\label{it:poly}  As pointed out by the referee, Boolean polyadic modal logics can be simulated by monadic ones \cite{Goguadze2003}. This simulation preserves and reflects  definability of frames. It may be the case that these simulation results extend to the many-valued framework. In that case, it would be enough to restrict to monadic languages $\lang'$ in order to obtain results about modal definability for classes of $\lang$-frames, where $\lang$ is an arbitrary polyadic modal language.  

\item The validity relations considered in this paper to define the notions of definability and $\lucas_n$-definability are based on models that evaluate formulas in a finite subalgebra of the standard MV-algebra $[0,1]$. Finding the right tools to generalize our results to a notion of definability based on a validity relation defined with $[0,1]$-valued models is a difficult task that would probably require new appropriate representation results for the variety of MV-algebras.

\item Similarly, studying modal definability for classes of relational structures in which the relations are many-valued is an important topic of further research. 

\item The validity relation $\models$ is obtained from $\models_n$ by restricting the set of possible valuations that can be added to an $\lang$-frame to turn it into a model. This paper illustrates the links that exist between these two validity relations. It would be interesting to develop tools to study (modal) definability in general situations involving  a validity relation which is a weakening of another one. 

\item Informally, the validity relation $\models$ defined in section \ref{sec:02} permit talking about the set of possible truth values in worlds of  $\lang_n$-frames. This gain of expressive power could turn out to be interesting for application-oriented modal extensions of many-valued logics such as the many-valued generalization of Propositional Dynamic Logic developed in \cite{Teheux2014}. 
\end{enumerate}
 
\section*{Acknowledgements}
We would like to thank the anonymous reviewer for his suggestions that helped to improve the readability of the paper and to shorten the proof of Theorem \ref{thm:main02}.

\appendix 
\section{Truth Lemma for $k$-ary modal operators}
In this appendix, we give the proof of Proposition \ref{lem:truth} for $k$-ary modalities. The case of a unary modality was given in \cite[Proposition 5.6]{Teheux2012}, but the general case has not been published previously.


We need to recall the following definitions and results from \cite{Bell1996}. Let $\alg{L}$ be a bounded distributive lattice and $k\geq 1$. Given $\bfx, \bfy \in \alg{L}^k$ we write $\bfx \approx \bfy$ if there is an $\ell\leq k$ such that $x_j=y_j$ for every $j\leq k$ with $j\neq \ell$. We write $\bfx | \bfy$ if there is a $\ell \leq k$ such that $x_\ell=y_\ell$. Moreover, if $F_j$ is a subset of $\alg{L}$ for every $j\leq k$, we let $F_1+ \cdots + F_k=\{\bfx \in \alg{L}^k \mid  \exists \ell \leq k \text{ such that }  x_\ell\in F_\ell\}$.

\begin{defn}[{\cite{Bell1996}}]
A subset $F$ of $\alg{L}^k$ is called a $\alg{L}^k$-\emph{filtroid} if it is increasing, contains $\{\bfone\}+\cdots + \{\bfone\}$, and contains $\bfx \wedge \bfy$ for every $\bfx,\bfy \in F$ such that $\bfx\approx\bfy$. A proper $\alg{L}^k$-filtroid $F$   is \emph{prime} if there are some prime filters $F_1, \ldots, F_k$ of $\alg{L}$ such that $F=F_1+\cdots +F_k$.
\end{defn}

For instance, for any $\alg{A}\in \var{MMV}_n^\lang$, any $u\in \alg{A}^+$, and any  $k$-ary modality $\nabla$,  the set \[\nabla^{-1}(u^{-1}(1)) \cap \idempo{\alg{A}}^k=\{\bfa \in \idempo{\alg{A}}^k \mid u(\nabla (\bfa))=1\}\] is a $\idempo{\alg{A}}^k$-filtroid.  

\begin{thm}[{\cite[Theorem 1.2]{Bell1996}}]\label{thm:prime} A proper $\alg{L}^k$-filtroid is the intersection of the prime $\alg{L}^k$-filtroids that contain it.
\end{thm}

\begin{defn}\label{defi:termes} 
Let $i\in\{1, \ldots, n\}$. We define the function $\tau_{i/n}\colon \lucas_{n}\to \lucas_{n}$ by
\[
\tau_{i/n}(x)=\begin{cases}
0 & x<\frac{i}{n},\\
1 & x\geq \frac{i}{n},\end{cases}
\]
and we  assume that $\tau_{{i}/{n}}$ is the interpretation on $\lucas_n$ of an algebraic term which is a composition of finitely many copies of the maps $x \mapsto x \oplus x$ and $x \mapsto x \odot x$  alone (the proof of existence of such a term appears in \cite{Ost88}).
\end{defn}

\begin{lem}
Let $\alg{A}\in \class{MMV}_n^\lang$ and $\nabla$ be a $k$-ary modality whose associated relation on $\alg{A}^+$ is $R$. For any $u, v_1, \ldots, v_k \in \alg{A}^+$, the following conditions are equivalent.
\begin{enumerate}[(i)]\label{lem:R}
\item\label{it:mlk01} $\bfv\in Ru$.
\item\label{it:mlk02} $u(\nabla (\bfa))\leq \bigvee_{i=1}^k v_i(a_i)$ for every $\bfa\in \alg{A}^k$.
\item\label{it:mlk03} $\big(v_1^{-1}(1)+\cdots +v_k^{-1}(1)\big)\cap \idempo{\alg{A}}^k$ is a prime $\idempo{\alg{A}}^k$-filtroid that contains $\nabla^{-1}(u^{-1}(1)) \cap \idempo{\alg{A}}^k$.
\end{enumerate}	
\end{lem}
\begin{proof}
(\ref{it:mlk02}) $\implies$ (\ref{it:mlk01}) follows by definition of $R$. 

(\ref{it:mlk01}) $\implies$ (\ref{it:mlk02}) Assume that there is some $\bfa \in \alg{A}^k$ and $\frac{i}{n}\in \lucas_n$ such that $v_\ell(a_\ell)< \frac{i}{n}\leq u(\nabla(\bfa))$ for every $\ell\leq k$. It follows that $v_\ell(\tau_{i/n}(a_\ell))=0$ for every $\ell \leq k$ while $u(\nabla(\tau_{i/n}(\bfa))=1$, a contradiction by definition of $R$.

(\ref{it:mlk01}) $\implies$ (\ref{it:mlk03}) follows by definition of $R$ and the concept of prime $\idempo{\alg{A}}^k$-filtroid.

(\ref{it:mlk01}) $\implies$ (\ref{it:mlk03}) is a consequence of Lemma \ref{lem:ba}.
\end{proof}

We give the proof of the polymodal version of  \cite[Proposition 5.6]{Teheux2012}.

\begin{proof}[Proof of Proposition \ref{lem:truth}]
We proceed by induction on the length of $\phi \in \Form_\lang$. The only non-trivial case is $\phi=\nabla(\phi_1, \ldots, \phi_k)$ for a $k$-ary modality $\nabla$. We let $\alpha=\alpha_\phi$  and $\alpha_\ell=\alpha_{\phi_\ell}$ for every $\ell\leq k$. We have to prove
\begin{equation}
u(\nabla(\alpha_1, \ldots, \alpha_k ))=\min\{\bigvee_{\ell=1}^k v_i(\alpha_\ell) \mid \bfv \in Ru\}.
\end{equation}
The inequality $\leq$ is (\ref{it:mlk01}) $\implies$ (\ref{it:mlk02}) in Lemma \ref{lem:R}. To prove the converse inequality, assume on the contrary that 
\[
u(\nabla(\alpha_1, \ldots, \alpha_k )) < \frac{i}{n}\leq \min\{\bigvee_{\ell=1}^k v_\ell(\alpha_\ell) \mid \bfv \in Ru\},
\]
for some $i\leq n$. It follows that $u(\nabla(\tau_{i/n}(\alpha_1), \ldots, \tau_{i/n}(\alpha_k )))=0$ while for every $\bfv\in Ru$ we have $v_\ell(\tau_{i/n}(\alpha_\ell))=1$ for some $\ell \leq k$. According to (\ref{it:mlk01}) $\iff$ (\ref{it:mlk03}) in Lemma \ref{lem:R}, this means that $(\tau_{i/n}(\alpha_1), \ldots, \tau_{i/n}(\alpha_k ))$ is in every prime $\idempo{\alg{A}}^k$-filtroid that contains $\nabla^{-1}(u^{-1}(1)) \cap \idempo{\alg{A}}^k$ but not in $\nabla^{-1}(u^{-1}(1)) \cap \idempo{\alg{A}}^k$. The contradiction follows from Theorem \ref{thm:prime}.
\end{proof}

\bibliographystyle{plain}

\end{document}